\documentclass{svproc}
\usepackage{graphicx}
\usepackage{amssymb}
\usepackage{url}
\usepackage{mathrsfs}
\usepackage{algorithm}
\usepackage[noend]{algpseudocode}
\usepackage{caption}
\usepackage{amsmath}

\let\oldleq\leq
 \usepackage{breqn}
\let\leq\oldleq

\newtheorem{Th}{Theorem}

\begin{document}
\mainmatter              
\title{Non-monotone Behavior of the Heavy Ball Method}
\titlerunning{Non-monotone Behavior of the Heavy Ball Method}  
%
\author{Marina Danilova\inst{1}, Anastasiya Kulakova\inst{2} and Boris Polyak\inst{3}}
\authorrunning{M. Danilova, A. Kulakova, B. Polyak} 
%
%

\institute{
Institute of Control Sciences RAS,
Laboratory of Adaptive and Robust Systems,
Profsoyuznaya, 65,
Moscow 117342,
Russia \\
\email{Danilovamarina15@gmail.com}
\and
Moscow Institute of Physics and Technology, 
Department of Control and Applied Mathematics,
Institutskiy per., 9,  
Dolgoprudniy 141700,
Russia \\
\email{anastasiya.kulakova@phystech.edu}
\and
Institute of Control Sciences RAS,
Laboratory of Adaptive and Robust Systems,
Profsoyuznaya, 65,
Moscow 117342,
Russia \\
\email{boris@ipu.ru}}

\maketitle              

\begin{abstract}
We focus on the solutions of second-order stable linear difference equations and demonstrate that their behavior can be non-monotone and exhibit peak effects depending on initial conditions. The results are applied to the analysis of the accelerated unconstrained optimization method --- the Heavy Ball method. We explain non-standard behavior of the method discovered in practical applications. In addition, such non-monotonicity complicates the correct choice of the parameters in optimization methods. We propose to overcome this difficulty by introducing new Lyapunov function which should decrease monotonically. By use of this function convergence of the method is established under less restrictive assumptions (for instance, with the lack of convexity). We also suggest some restart techniques to speed up the method's convergence.
\end{abstract}
\section{Introduction}
It is well known that $n$-th order scalar linear difference equations
\[
x_k + a_1x_{k-1} + \dots + a_n x_{k-n} = 0,\ k = n,\ n+1,\ \dots;\ a_i\in \mathbb{R}
\]
with initial conditions
$$x^{(0)} = (x_0, \dots, x_{n-1})\in \mathbb{R}^n$$
are stable, i.e. $\lim\limits_{k\rightarrow\infty} x_k=0$,  if and only if the moduli of the roots $\lambda_i$ 
of the characteristic polynomial
\begin{equation*}\label{poly}
p(\lambda) = \lambda^n + a_1 \lambda^{n-1} + \dots + a_{n-1} \lambda + a_n
\end{equation*}
are less than 1 \cite{El}. 

However, the convergence to zero can be non-monotone. This effect can  be described by the following quantity:
\begin{equation*}\label{eta}
\eta(x^{(0)}) = \max_{k=n, n+1, \dots}|x_k|,
\end{equation*}
which will be referred to as peak of the solution (provided that $\eta(x^{(0)}) > 1$) for a given root location $\lambda$ and initial condition $x^{(0)}$. Without loss of generality we assume that $\|x^{(0)}\|_{\infty}\leq 1$. The estimates of $\eta(x^{(0)})$ are demonstrated in the recent paper \cite{shch2018}.

The main objective of the present paper is to link the peak effects in linear difference equations  with the non-monotone behavior of such unconstrained optimization methods as the Heavy Ball method~\cite{Polyak_1,polyak1987}, Nesterov's accelerated gradient method~\cite{nest2004} and, for example, the recently proposed triple momentum method~\cite{fast2018}.
When applied to a quadratic function, momentum methods are described by second-order linear matrix difference equations. In numerous simulations it was found out that the methods demonstrate a non-monotone convergence to a minimum~\cite{boyd2014}.

Below we restrict ourselves with the analysis of  the Heavy Ball method, but similar techniques can be extended to other accelerated optimization methods. It is worth mentioning that such methods are currently widespread and implemented to minimization problems occurring in neural networks. That is why a detailed study of these methods is of a great importance.
In addition, the numerous restart techniques~\cite{Restart_2} gain their popularity. To design the restart technique it is important to know how large the deviations from the solution should be to make an assumption about incorrect parameters' choice and start the method from the beginning at the current point.

The paper
is organized as follows. In Section \ref{lde_sec} we provide the results on peak effects in second-order linear difference equations. Next Section 3 contains applications of these results to the Heavy Ball method and demonstrates the dependence of peak value on initial conditions. In Section 4 a new Lyapunov function 
is constructed; it decreases monotonically in contrast with the objective function or distance to the solution. By use of this function we are also able to prove global convergence for non-convex functions (under less restrictive Polyak-$\L$ojasiewicz condition). Conclusions and future directions for research are summarized in Section 5. 

\section{Peak effect for second-order difference equations}\label{lde_sec}

We consider the  case of second-order difference equation ($n=2$), which can be written down in the form
\begin{equation}\label{lde_2}
x_k = a_1x_{k-1} + a_2 x_{k-2},\ k = 2,\ 3,\ \dots;\ a_1,\ a_2 \in \mathbb{R} 
\end{equation}
with initial conditions 
$$x^{(0)} = (x_0, x_1)\in \mathbb{R}^2$$
and the characteristic polynomial
\begin{equation}\label{poly_2}
p(\lambda) = \lambda^2 - a_1 \lambda - a_2. 
\end{equation}

For the second-order difference equation  the following stability domain (Fig.~\ref{domain}) in the space $(a_1, a_2)$ is well known \cite{El}.

\begin{figure}[h!]
\center{\includegraphics[width=0.8\linewidth]{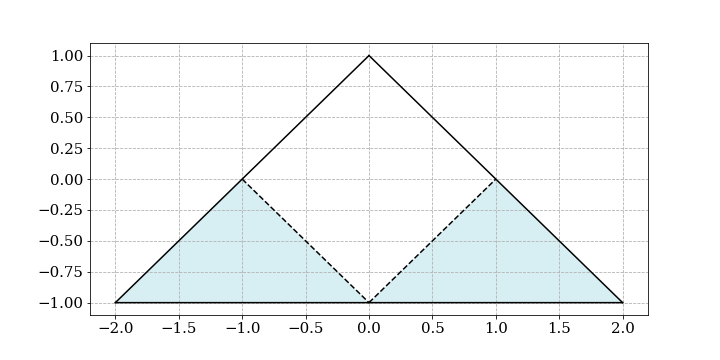}}
\caption{The stability domain of the second-order difference equation, the shaded area corresponds to peak effect.}
\label{domain}
\end{figure}

Our aim is to describe possible non-monotone behavior of stable solutions and, in particular, to measure $\eta(x^{(0)})$.

Two cases can be considered: the first corresponding to equal roots of the characteristic polynomial and the second with different roots. 

In the case of equal roots, i.e. $\lambda^2 - a_1 \lambda - a_2 = (\lambda - \rho)^2 = \lambda^2 - 2\lambda\rho + \rho^2,\ 0<\rho<1$, and (\ref{lde_2}) reads:
\begin{equation*}\label{rho_root}
x_k = 2\rho x_{k-1} - \rho^2x_{k-2}.
\end{equation*}

The following expression for solution can be easily derived:
\begin{equation}\label{xk_rho}
x_k = x_1k\rho^{k-1} - x_0(k-1)\rho^k.
\end{equation}

In this case the non-monotone behavior can be observed (Fig.~\ref{double}).

\begin{figure}[h!]
\center{\includegraphics[width=0.8\linewidth]{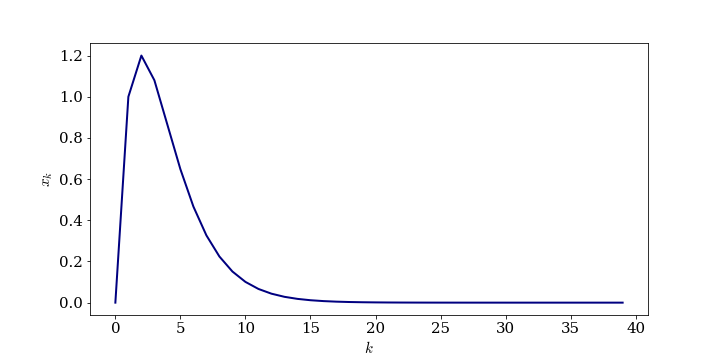}}
\caption{The iterative process with $\rho = 0.6$ and $x^{(0)} = (0, 1)$.}
\label{double}
\end{figure}

We conclude that for all $k\geq 2$ 
\begin{equation}\label{peak_2}
\max\limits_{\parallel x^{(0)}\parallel_{\infty}\leq1}x_k = k\rho^{k-1} + (k-1)\rho^k.
\end{equation}

This maximum is achieved for $x^{(0)} = (-1, 1)$. 

Now let's derive $k_{max} = \mathrm{argmax}\left(k\rho^{k-1} + (k-1)\rho^k \right) $ and $\eta(x^{(0)})$. 
By differentiation and setting the derivative value equal to zero we obtain an expression for $k_{max}$:
\begin{equation}\label{k_max}
k_{max} = \left\lceil \frac{\rho \ln \rho - \rho - 1}{\ln \rho (1 + \rho)} \right\rceil.
\end{equation}

The value of the peak can be obtained by a substitution (\ref{k_max}) in the formula~(\ref{peak_2}). 

For $\rho \to 1$ we get
\begin{equation*}
\eta(x^{(0)})\approx \frac{2}{e(1-\rho)}.
\end{equation*}

Thus large deviations can arise for some initial conditions. However, some pairs $(x_0, x_1)$ do not imply peak effects. For instance, for $x_0=1, x_1=1$ we obtain that $|x_k|< 1, \ k=2, \dots $.

Considering the case of different real roots $\lambda_1, \lambda_2$ of the characteristic polynomial~(\ref{poly_2}), we notice the following:

$\bullet$ If $\lambda_2 > \lambda_1 > \rho, \ 0<\rho<1$,  than there exist such initial conditions $x^{(0)}$, that the trajectory $x_k$ will behave non-monotonically and its peak will be located higher than in the case of equal roots (Fig.~\ref{more}).
\begin{figure}[h]
\center{\includegraphics[width=0.8\linewidth]{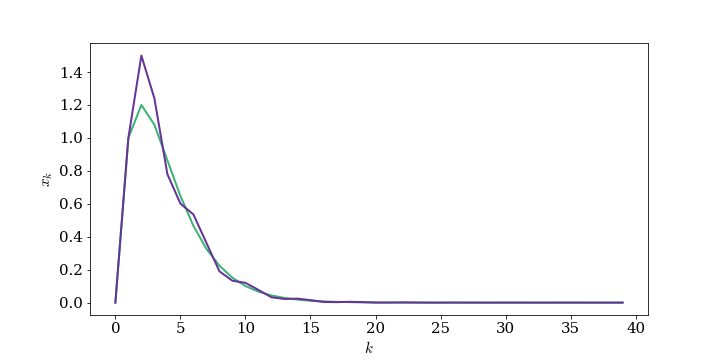}}
\caption{The trajectories of the iterative processes with initial conditions $x^{(0)} = (0, 1)$ and $\rho = 0.6$ (green) and $\lambda_1 = 0.7,\ \lambda_2 = 0.8$ (purple).}
\label{more}
\end{figure}

$\bullet$ If $\lambda_1 < \lambda_2 < \rho, \ 0<\rho<1$, the peak of $x_k$ (if any) will be located lower than in the case of equal roots $\rho$ for all initial conditions (Fig.~\ref{less}).
\begin{figure}[h]
\center{\includegraphics[width=0.8\linewidth]{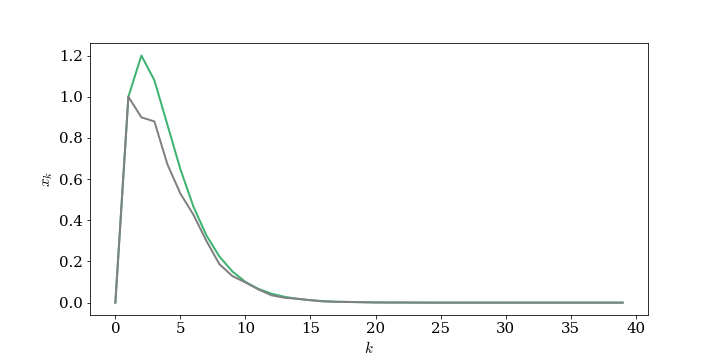}}
\caption{The trajectories of the iterative processes with initial conditions $x^{(0)} = (0, 1)$ and $\rho = 0.6$ (green) and $\lambda_1 = 0.4,\ \lambda_2 = 0.5$ (gray).}
\label{less}
\end{figure}

The proof of both statements is given in~\cite{shch2018}.

\section{Analysis of the Heavy Ball method }
\subsection{The Heavy Ball method}
We consider the simplest unconstrained optimization problem
\begin{equation}\label{opt_pr}
\min\limits_{x\in \mathbb{R}^n} f(x),
\end{equation}
where $f(x): \mathbb{R}^n \to \mathbb{R}$ is a smooth objective function, $x^*$~--- the minimum point. We restrict our analysis with the case of
 quadratic objective function:
\begin{equation*}\label{opt_qu}
\min\frac{1}{2}(Ax,x)-(b,x),\ x,\ b \in \mathbb{R}^n 
\end{equation*}
with $A\in \mathbb{R}^{n \times n}$ being positive-definite matrix $A\succ 0$, $\ \nabla f(x) = Ax - b$, $\ x^* = A^{-1}b$, $L$ and $\mu>0$~--- the maximum and minimum eigenvalues of $A$ respectively. It means that we focus on strongly convex, smooth case of the objective function.

Without  loss of generality we can assume that after substitution $\hat{x} = x - x^*$ objective function $f \ $ has the
following form:
\begin{equation*}\label{opt_qu1}
f(x) = \frac{1}{2}(A\hat{x},\hat{x})   
\end{equation*}
So, the optimal point is $x^* = 0$, $f^* = 0$.

There are numerous iterative methods to solve this problem;   gradient methods are among the most popular, see e.g. \cite{polyak1987,nest2004}. The behavior of gradient methods is simple enough: they exhibit monotone convergence both for objective function and distance to the minimum point. The situation with accelerated first-order methods is much more complicated. We will focus on one of them --- so called Heavy Ball method proposed in \cite{Polyak_1}:
\begin{equation}\label{mtsh}
x_{k+1}=x_k-\alpha \nabla f( x_k)+\beta (x_k-x_{k-1})=x_k-\alpha A x_k+\beta (x_k-x_{k-1})
\end{equation}

Here  a momentum term was added to the classic gradient method, which accelerated the convergence and made the trajectory look like a smooth descent to the bottom of the ravine, rather than zigzag. The traditional choice of initial condition is
\begin{equation}\label{x0}
x_1=x_0,
\end{equation}
i.e. the first iteration is just a gradient step.
It is known from \cite{Polyak_1} that for 
\begin{equation*}\label{al_bet}
0\le \beta <1, \quad 0< \alpha < \frac{2(1+\beta)}{L}
\end{equation*}
there is a convergence to the solution with linear rate, $||x_k||= O(q^k), \ q<1$.
The optimal parameters $\alpha$ and $\beta$ providing the fastest convergence $q=\frac{\sqrt{L}-\sqrt{\mu}}{\sqrt{L}+\sqrt{\mu}}$ are also known:
\begin{equation}\label{al_mtsh}
\alpha=\frac{4}{(\sqrt{L}+\sqrt{\mu})^2},
\quad
\beta= \left(\frac{\sqrt{L}-\sqrt{\mu}}{\sqrt{L}+\sqrt{\mu}}\right)^2 = q^2.
\end{equation}

However, the non-asymptotic behavior of the method with optimal parameters for a simple example is shown on Fig. \ref{oscill_o} and with parameters very close to optimal --- on Fig. \ref{oscill}. In all examples below matrix $A$ is taken diagonal.

\begin{figure}[h!]
\center{\includegraphics[width=\linewidth]{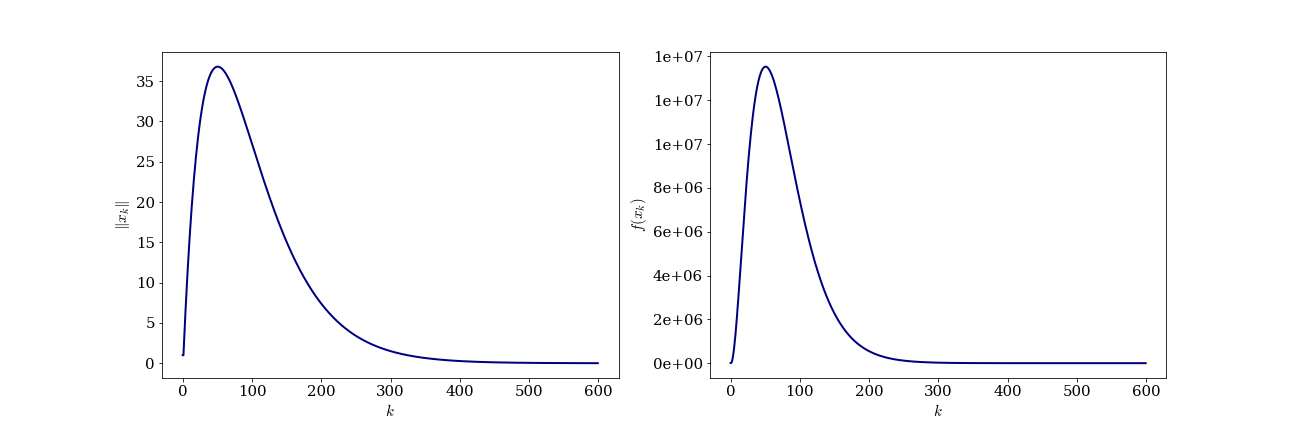}}
\caption{Non-monotone behavior of the Heavy Ball method with $x_0 = [0, 0, 0, 1]$,\ $\mu=1$,\ $L=10^4$,\ $\alpha,\ \beta$~--- optimal.}
\label{oscill_o}
\end{figure}

\begin{figure}[h!]
\center{\includegraphics[width=\linewidth]{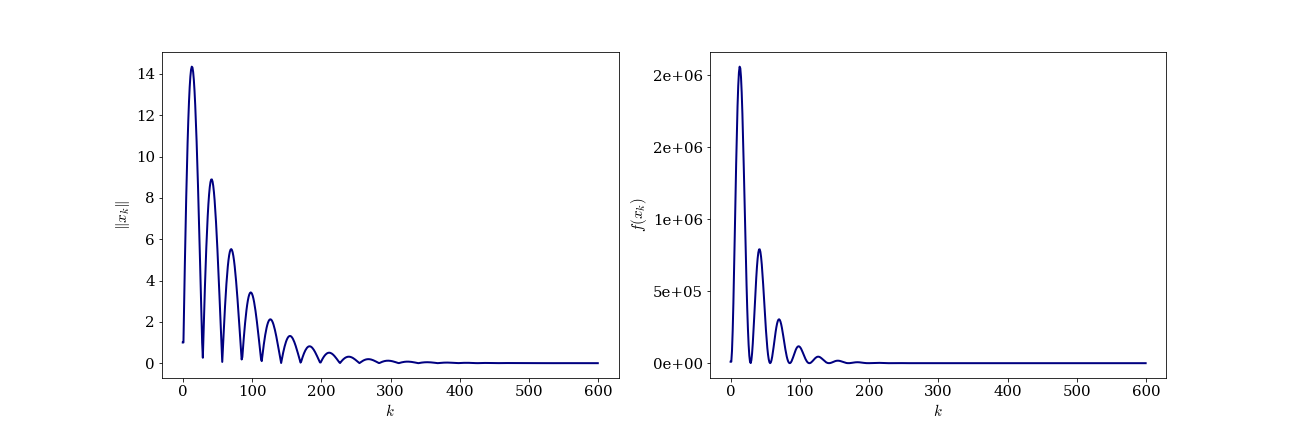}}
\caption{Non-monotone behavior of the Heavy Ball method with $x_0 = [0, 0, 0, 1]$,\ $\mu=1$,\ $L=10^4$,\ $\alpha,\ \beta$~--- close to optimal.}
\label{oscill}
\end{figure}

We conclude that the method exhibits strongly non-monotone 
behavior.
\subsection{Convergence analysis}
To explain the behavior  of the Heavy Ball method (\ref{mtsh}) with optimal $\alpha$ and $\beta$ (\ref{al_mtsh}) written in the form of second-order difference equation we consider it component-wise.

Let's start with the coordinate $x^1= (x,e_1)$, which corresponds to the minimal eigenvalue $\lambda_{\min} = \mu, \ A e_1=\mu e_1$. The  method for this  coordinate in the form of scalar linear difference equation along with its characteristic polynomial is written down below:

\begin{equation*}\label{mtsh_x1}
x^1_{k+1}=(1-\alpha \mu + \beta) x^1_k - \beta x^1_{k-1},
\end{equation*}

\begin{equation*}\label{poly_x1}
\rho^2 - (1-\alpha \mu + \beta) \rho + \beta = 0.
\end{equation*}
It is easily determined that a characteristic polynomial has both roots equal to $q$, meaning that
$\rho=q = \left( \frac{\sqrt{L} - \sqrt{\mu}}{\sqrt{L} + \sqrt{\mu}}\right).$ So the general solution is given by expression~(\ref{xk_rho}), while maximum provided by formula (\ref{peak_2}) with obvious change of notation.

Moving on to the coordinate $x^n=(x,e_n), \ A e_n=L e_n$, which corresponds to the maximum eigenvalue $\lambda_{\max} = L$, we notice that the only difference from the previous case is in the sign of roots of the characteristic polynomial, i.e. $\rho=-q$. However this implies different behavior of solutions, even for $x_0=x_1=1$ the trajectory is oscillating with possible large deviations. From formula (\ref{xk_rho}) we conclude that initial conditions $x_0=x_1=1$ cause the largest (in absolute value) peak effect equal to (\ref{peak_2}).

Now we consider a more general case of the coordinate $x^i=(x,e_i), \ 2\le i\le n-1, \  Ae_i=\lambda e_i, \ \mu<\lambda<L$. The characteristic polynomial
\begin{equation*}\label{poly_x2}
\rho^2 - (1-\alpha \lambda + \beta) \rho + \beta = 0
\end{equation*}
has complex roots
\begin{equation*}\label{x2_roots}
\rho_{1,2}=\frac{\left( \sqrt{L-\lambda} \pm i \sqrt{\lambda - \mu} \right)^2}{(\sqrt{L} + \sqrt{\mu})^2},\quad
|\rho|=\frac{\sqrt{L} - \sqrt{\mu}}{\sqrt{L} + \sqrt{\mu}} = q. 
\end{equation*}

The general solution is written down below:
\begin{equation*}\label{x2_sol}
x^i_k=\left[C_1 \cos(\omega k) + C_2 \sin(\omega k) \right]q^k, 
\end{equation*}
where
\[
\sin(\omega) = \frac{2 \sqrt{\lambda - \mu} \sqrt{L - \lambda}}{(L - \mu)}, \quad \cos(\omega) = \frac{L + \mu - 2 \lambda}{(L - \mu)}
\]

\[
C_1=x^2_0,\quad 
C_2=\frac{x^2_1(\sqrt{L}+\sqrt{\mu})^2 - x^2_0 (L+\mu-2\lambda)}{2\sqrt{\lambda-\mu}\sqrt{L-\lambda}}. 
\]
 The trajectory (for $n=3$) demonstrates  oscillations  (Fig. \ref{icdea_6}). 

\begin{figure}[h]
		\noindent
		\centering
		\includegraphics[width=0.8\linewidth]{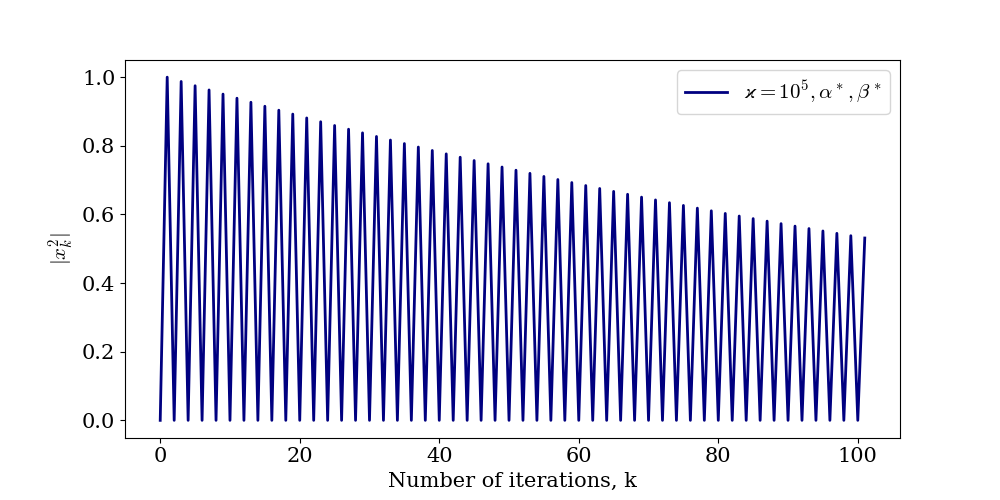}
	\caption{Dependence of the coordinate $|x^2| \ $ on the number of iterations $k$ under  conditions $x_0 = 0_n, \ x_1 = 1_n$.}
\label{icdea_6}
\end{figure}

\subsection{Peak effect}
From our previous observations in Section \ref{lde_sec} we note that peak effects in linear difference equations are common and depend on initial conditions. For the worst case the following proposition holds (we provide the estimates for the most important case of large condition number $\kappa=L/\mu$).

\begin{Th}
Assume that $f(x) = \frac12 \left( Ax, x\right),$ where $A \in \mathbb{R}^{n \times n}, \ A = A^{\top} \succ 0$. We have strong convexity parameter $\mu = \lambda_{\min} > 0$ and $L=\lambda_{\max}$, where $\lambda_{\min}$ and $\lambda_{\max}$ are the minimum and maximum eigenvalues of $A$, respectively. Then  initial conditions $x_0 =-e_1, \ x_1 =e_1\in \mathbb{R}^n, \ \|x_0\| = \|x_1\|=  1$ cause peak effect in the Heavy Ball method with optimal parameters $\ \alpha^*, \beta^*$:
\begin{equation*}
        \max\limits_{k}||x_k|| \geq \frac{\sqrt{\kappa}}{2 e},
\end{equation*}
while standard (\ref{x0}) initial conditions  $x_0 = x_1 =e_n\in \mathbb{R}^n, \ \|x_0\| = \|x_1\|=  1$  cause the same peak effect combined with oscillating behavior.      
        
\end{Th}
    
The proof follows from the estimates obtained in the previous section. Of course, other initial conditions also may lead to non-monotone behavior of iterations; we indicate the ones which provide the largest deviations from the minimum point.

The figures below show that the Heavy Ball method exhibits the non-monotone behavior on the test problem $n=3, \ \kappa =10^4$, namely, a sharp increase in the function under various initial conditions.

\begin{figure}[H]
		\noindent
		\centering
		\includegraphics[width=\linewidth]{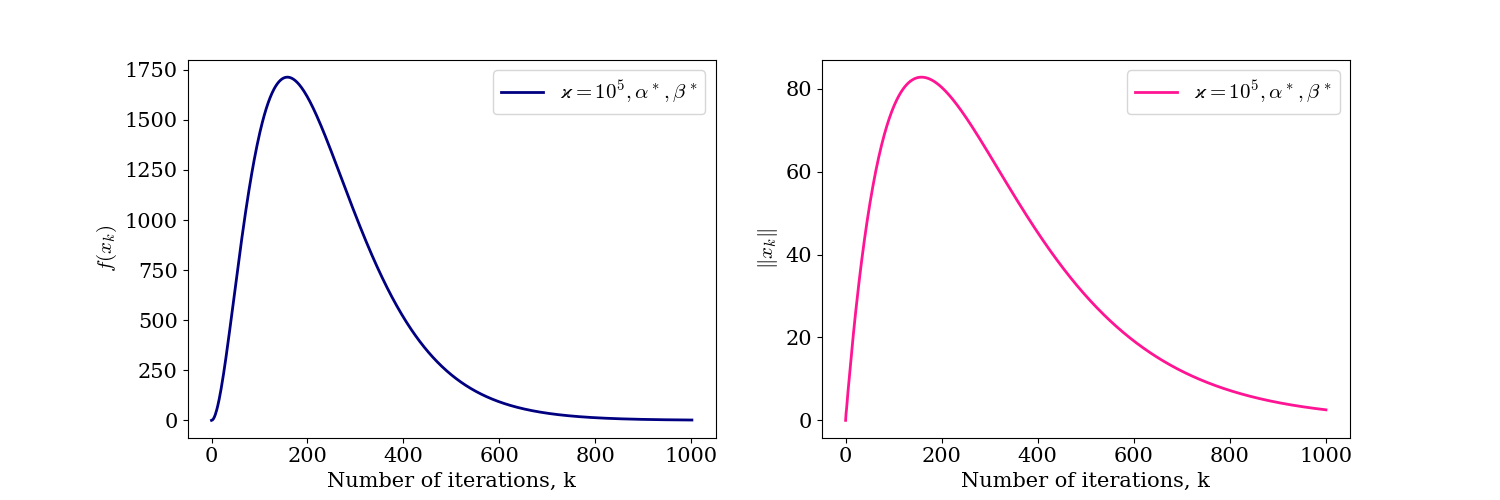}
	\caption{Dependence of the objective function $f(x_k) \ $ and $\|x_k\|$ on the number of iterations $k$ under the initial conditions $x_0 = 0_n, \ x_1 = 1_n$.}
\end{figure}
\begin{figure}[H]
		\noindent
		\centering
		\includegraphics[width=\linewidth]{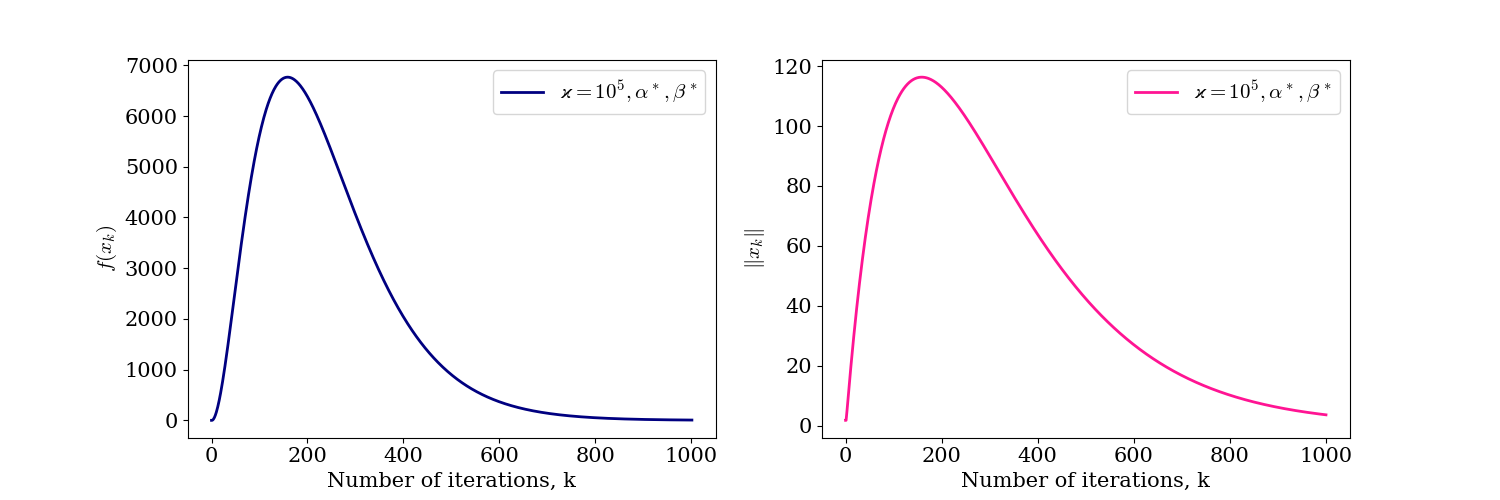}
	\caption{Dependence of the objective function $f(x_k) \ $ and $\|x_k\|$ on the number of iterations $k$ under the initial conditions $x_0 = x_1 = 1_n$.}
\end{figure}

To sum up, we applied the results on linear difference equations to the Heavy Ball method analysis. It was shown that even the choice of optimal parameters and standard initial conditions can not guarantee the monotone convergence. 

\section{Lyapunov function for the Heavy ball method}
In this section we extend the analysis of the Heavy Ball method (both for continuous and discrete versions) to non-quadratic objective functions via Lyapunov function technique. Thus we treat the unconstrained optimization problem
\begin{equation*}\label{min}
\min\limits_{x\in \mathbb{R}^n} f(x),    
\end{equation*}
where $f(x)$ is differentiable function bounded from below: $f(x)\ge f^*$. Notice that here we do not assume neither convexity nor strong convexity of $f(x)$.

\subsection{Construction of the Lyapunov function}

 Lyapunov functions is a common tool for proving the stability of nonlinear systems described by differential or difference equations. The Lyapunov function is a scalar function that decreases monotonically in stable system. 

The Heavy Ball method, as we verified above, does not exhibit monotone behavior even for the simplest case of quadratic function. Thus neither $f(x)$ nor $\|x-x^*\|$ can be used as the Lyapunov function. We suggest new Lyapunov function that can  help to select the parameters of the method and overcome the difficulties related to its non-monotone transient process.

\subsubsection{Continuous case.}

Before moving to the construction of the Lyapunov function for discrete case, we would like to consider continuous case from  \cite{Polyak_1}:
\begin{equation}
\label{eq4}
\ddot{x} + a\dot{x}+b\nabla f(x)=0.
\end{equation}

In mechanical interpretation, $x(t)\in  \mathbb{R}^n$ is the trajectory of the body (''heavy ball''), $\dot{x}, \ddot{x}$ are its velocity
and acceleration, $a>0, \ b>0$ are scalar parameters while $f(x)\ge f^*$ is the potential energy and $\nabla f(x)$ is its gradient. It is known from \cite{Polyak_1,Polyak_Sherbakov} that  $\nabla f(x(t))$ tends to zero, however the convergence can be non-monotone. Our goal is to obtain the upper bounds for the convergence.

Firstly, let's rewrite (\ref{eq4}):
\begin{eqnarray*}
&&\dot{x}=y,\\ 
&&\dot{y}=-ay-bf'(x) \nonumber
\end{eqnarray*}
with some initial conditions $x(0), \ y(0)$.
According to paper \cite{Polyak_Sherbakov}, $V(x,y)$ can be chosen from mechanical analogies. Consequently, it is possible to represent the function $V(x,y)$ in the form of the total energy of the system:
\begin{equation*}
V(x,y)=f(x)+\frac{1}{2b}\|y\|^2. 
\end{equation*}

For the time derivative we have
\begin{equation*}
\dot{V}(x,y) = \bigl(f'(x), y \bigr) + \frac{1}{2b}2\bigl(y, -ay-bf'(x) \bigr) = -\frac{a}{b}\|y\|^2\le 0. 
\end{equation*}

Thus we get an upper bound for $f(x(t))$:
\begin{equation*}\label{1}
 f(x(t))-f^*\le f(x(0))-f^*+\frac{1}{2b}\|y(0)\|^2.
\end{equation*}

In particular, for zero initial velocity $ f(x(t))-f^*\le f(x(0))-f^*$.

\subsubsection{Discrete case.}   
Now we proceed to discrete-time version of the Heavy Ball method
\begin{equation}
\label{discrete}
x_{k+1}=  x_{k}  -\alpha \nabla f(x_{k}) + \beta (x_{k} - x_{k-1}), 
\end{equation}
and assume additionally that $f$ is $L$-smooth:
\begin{equation*}
\label{L}
||\nabla f(x)-\nabla f(y)||\le L||x-y||. 
\end{equation*}
\begin{Th} \label{th2}
Assume that
\begin{equation}
\label{ab}
0< \alpha < \frac{1}{L}, \; 0\le \beta \le \sqrt{1 - \alpha L}.
\end{equation}
Then for any initial conditions $x_0, \ x_1 \in \mathbb{R}^n$ the function
\begin{equation}
\label{Lyap_function}    
V_k = f(x_k)-f^* + \frac{1-\alpha L}{2 \alpha} \|x_k - x_{k-1}\|^2
\end{equation}
is the Lyapunov function for the discrete case of the Heavy Ball method, that is $V_k\le V_{k-1}$.
\end{Th} 

\begin{proof}

For  iterations (\ref{discrete}), we have
\begin{equation}
\label{11}
x_k - x_{k-1} = -\alpha \nabla f(x_{k-1}) + \beta (x_{k-1} - x_{k-2}), 
\end{equation}

\begin{eqnarray}\label{12}
\|x_k - x_{k-1}\|^2 = \alpha^2 \|\nabla f(x_{k-1})\|^2 + \beta^2 \|x_{k-1} - x_{k-2}\|^2 \\ \nonumber
- 2 \alpha \beta \left( \nabla f(x_{k-1}), x_{k-1} - x_{k-2}\right).
\end{eqnarray}

Since function $f \in \mathscr F^{1,1}_{L} \ $, it has the Lipschitz gradient, so following equation from \cite{nest2004} can be applied
\begin{equation}
\label{13}
f(x_{k}) \leq f(x_{k-1}) + \left\langle \nabla f(x_{k-1}), x_{k} - x_{k-1}\right\rangle + \frac{L}{2} \|x_{k} - x_{k-1}\|^2.
\end{equation}

Adding (\ref{11}), (\ref{12}) to (\ref{13}), we get
\begin{eqnarray}\label{14}
\lefteqn{f(x_{k}) \leq f(x_{k-1}) - \alpha \|\nabla f(x_{k-1})\|^2 +} \\ \nonumber
& & \beta \left\langle \nabla f(x_{k-1}), x_{k-1} - x_{k-2}\right\rangle
+ \frac{\alpha^2 L}{2} \|\nabla f(x_{k-1})\|^2 + \\ \nonumber
& & \frac{\beta^2 L}{2} \|x_{k-1} - x_{k-2}\|^2  - L \alpha \beta \left( \nabla f(x_{k-1}), x_{k-1} - x_{k-2}\right).
\end{eqnarray}

Multiplying (\ref{12}) by $\frac{1-\alpha L}{2 \alpha}$
and adding to (\ref{14}) we obtain
\[
f(x_{k}) + \frac{1-\alpha L}{2 \alpha}  \|x_k - x_{k-1}\|^2  \; \leq \; f(x_{k-1}) - \alpha \|\nabla f(x_{k-1})\|^2 + 
\]
\[
\beta \left\langle \nabla f(x_{k-1}), x_{k-1} - x_{k-2}\right\rangle +
\frac{\alpha^2 L}{2} \|\nabla f(x_{k-1})\|^2 + \frac{\beta^2 L}{2} \|x_{k-1} - x_{k-2}\|^2 -
\]
\[
L \alpha \beta \left( \nabla f(x_{k-1}), x_{k-1} - x_{k-2}\right) + \frac{1-\alpha L}{2 \alpha}\cdot
\]
\[
\left( \alpha^2 \|\nabla f(x_{k-1})\|^2 + 
\beta^2 \|x_{k-1} - x_{k-2}\|^2 - 
2 \alpha \beta \left( \nabla f(x_{k-1}), x_{k-1} - x_{k-2}\right)\right).
\]

Collecting terms yields
\[
f(x_{k}) + \frac{1-\alpha L}{2 \alpha}  \|x_k - x_{k-1}\|^2  \; \; \leq \; \; f(x_{k-1}) + \left(  \frac{\beta^2 L}{2}  + \frac{(1-\alpha L)\beta^2}{2 \alpha} \right)\cdot
\]
\[
\|x_{k-1} - x_{k-2}\|^2 +
\left( \frac{\alpha^2 L}{2} - \alpha + \frac{(1- \alpha L) \alpha}{2} \right) \|\nabla f(x_{k-1})\|^2 +
\]
\[
\left(\beta - L \alpha \beta  -\beta + \alpha L \beta \right) \left( \nabla f(x_{k-1}), x_{k-1} - x_{k-2}\right).
\]

As a result,
\[ 
f(x_{k}) + \frac{1-\alpha L}{2 \alpha}  \|x_k - x_{k-1}\|^2  \; \; \leq \; \; f(x_{k-1}) + \left(  \frac{\beta^2 L}{2}  + \frac{(1-\alpha L)\beta^2}{2 \alpha} \right)\cdot  
\]
\[
\|x_{k-1} - x_{k-2}\|^2 +
 \left( \frac{\alpha^2 L}{2} - \alpha + \frac{(1- \alpha L) \alpha}{2} \right) \|\nabla f(x_{k-1})\|^2 
\]

\begin{itemize}
    \item
\[ \frac{\beta^2 L}{2}  + \frac{(1-\alpha L)\beta^2}{2 \alpha}  \; = \; \frac{\beta^2 L \alpha  + \beta^2 - \beta^2 \alpha L}{2 \alpha} \; = \; \frac{\beta^2}{2 \alpha} \; > \; 0 
\]
    \item
\[
    \frac{\alpha^2 L}{2} - \alpha + \frac{(1- \alpha L) \alpha}{2} = \frac{\alpha^2 L - 2 \alpha + \alpha - \alpha^2 L}{2} = - \frac{\alpha}{2} \; < \; 0
    \]
\end{itemize}

\begin{eqnarray}\label{15}
f(x_{k}) + \frac{1-\alpha L}{2 \alpha}  \|x_k - x_{k-1}\|^2  \; \; \leq \; \; f(x_{k-1}) + \left(  \frac{\beta^2}{2 \alpha}\right)  \|x_{k-1} - x_{k-2}\|^2 + \\ \nonumber
\left( - \frac{\alpha}{2} \right) \|\nabla f(x_{k-1})\|^2
\end{eqnarray}

Since $\left( - \frac{\alpha}{2} \right) \|\nabla f(x_{k-1})\|^2 \; < \; 0 \ $
and having (\ref{ab}) we arrive to the desired inequality $V_k\le V_{k-1}$.

\end{proof}
Theorem \ref{th2} provides the upper bound of $f(x_k)$ for the Heavy Ball method:
\begin{equation*}\label{f_k}
f(x_k)-f^*\le f(x_0)-f^*+\frac{1-\alpha L}{2\alpha}||x_1-x_0||^2   
\end{equation*}
and for standard initial condition $x_1=x_0$ we obtain 
\begin{equation*}\label{f_kst}
f(x_k)-f^*\le f(x_0)-f^*.    
\end{equation*}
Notice the lack of such bound for more narrow class of strongly convex quadratic functions, on the other hand this estimate holds for more restrictive conditions on parameters $\alpha, \ \beta$, see (\ref{ab}).
    
\subsubsection{Numerical experiments.}

We have obtained the following Lyapunov function:
\[
V_k = f(x_k) + \frac{1-\alpha L}{2 \alpha} \|x_k - x_{k-1}\|^2
\]
with the restriction on the parameters (\ref{ab}).

We would like to show numerically that the function proposed above is indeed the Lyapunov function. As a particular example a quadratic function $f(x), \ n = 20, \ \kappa = 10^5 \ $ is taken along with the Heavy Ball method with a set of parameters $\alpha, \ \beta$. It can be observed (Fig. \ref{mono}) that the Lyapunov function decreases monotonically on the trajectory of the method.

\begin{figure}[h]
    \centering
    \includegraphics[width=\linewidth]{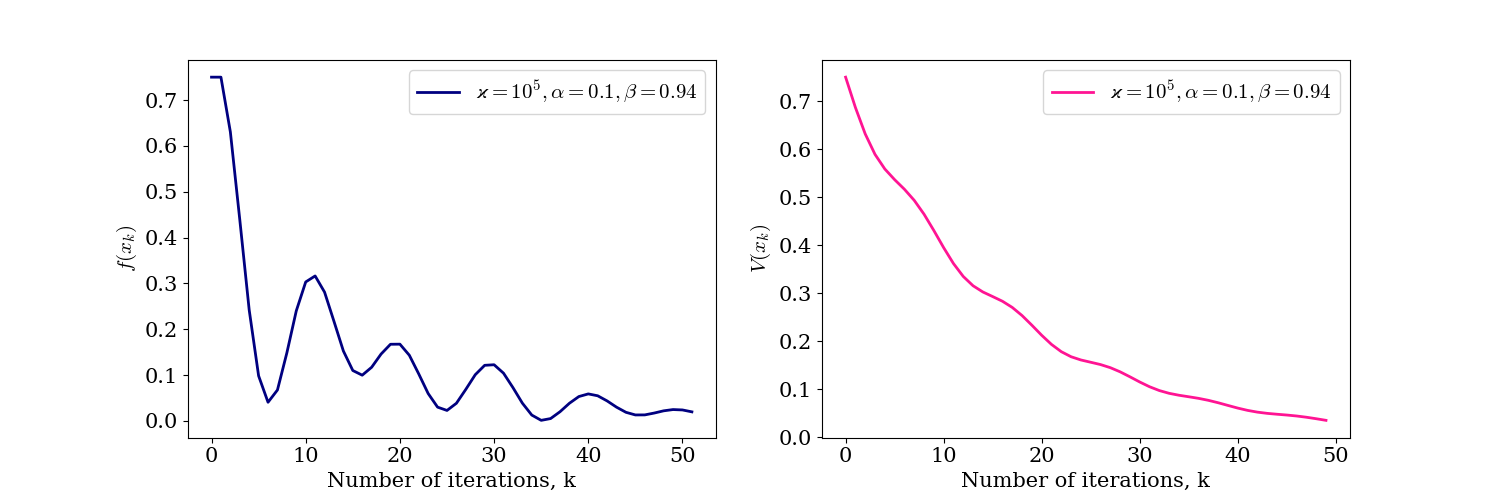}
    \centering
    \includegraphics[width=\linewidth]{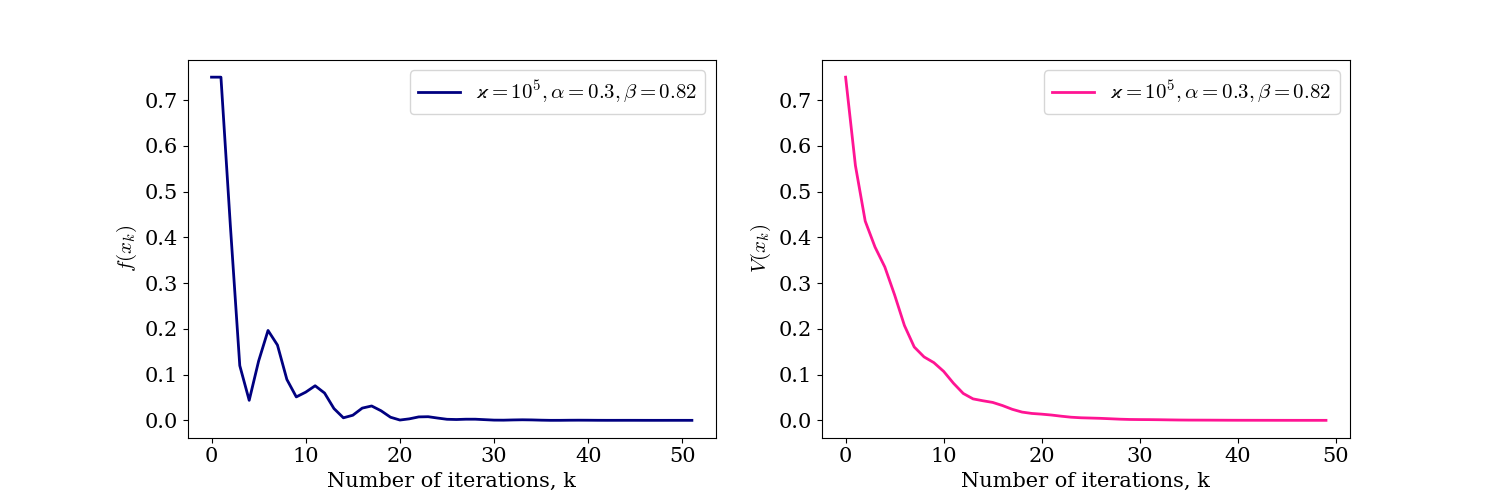}
    \centering
    \includegraphics[width=\linewidth]{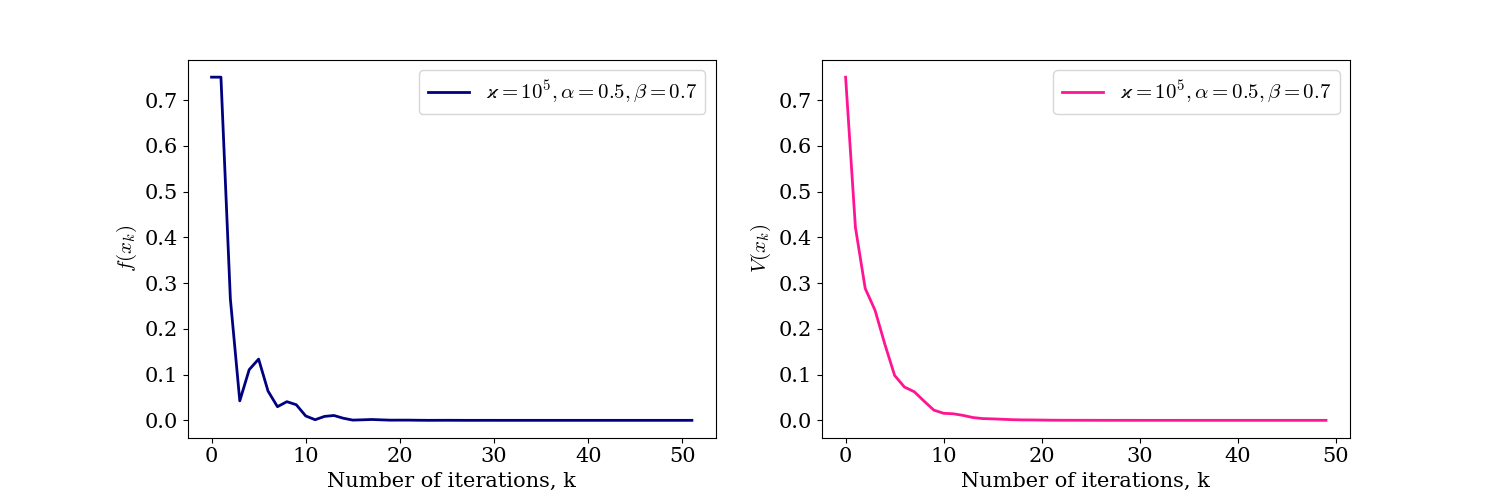}
    \caption{\label{mono}The behavior of objective function $f(x_k)$ and Lyapunov function $V(x_k)$ with different parameters $\alpha, \ \beta \ $ and the condition number $\kappa = 10^5$.}
\end{figure}

\subsection{Global convergence}
In the statement above we did not prove neither convergence of $f(x_k)$ to $f^*$ nor convergence $V_k$ to zero. To obtain such results further assumptions on $f(x)$ are needed. The least restrictive condition is 
\begin{equation}\label{PL}
||\nabla f(x)||^2\ge 2\mu (f(x)-f^*), \mu>0    
\end{equation}
for all $x\in \mathbb{R}^n$. This inequality is satisfied for strongly convex functions \cite{polyak1987,nesterov_1}, but in general it does not require convexity. The condition has been proposed in \cite{pol63}, sometimes it is called Polyak-$\L$ojasiewicz condition. 
\begin{Th} 
If (\ref{PL}) holds and 
\[\; \alpha \in \left( 0, \frac{1}{L}\right), \; \beta \in  \left[ 0, \sqrt{(1 - \alpha L)(1 - \alpha \mu)} \right].
\]
then for any initial conditions $x_0, x_1 \in \mathbb{R}^n$ Lyapunov function converges linearly
\begin{equation*}
\label{rate}
V_k\le V_0 q^k,\quad q=1-\alpha\mu <1,
\end{equation*}
while for $x_0 =x_1 $ objective function converges linearly
\begin{equation*}
\label{rate}
f(x_k)-f^*\le (f(x_0)-f^*) q^k.
\end{equation*}
\end{Th} 

\begin{proof}

It was shown above (\ref{15}), that
\[
f(x_k) - f^* + \frac{1-\alpha L}{2 \alpha} \|x_k - x_{k-1}\|^2 \; \leq  \; f(x_{k-1}) - f^* + \frac{\beta^2}{2 \alpha} \|x_{k-1} - x_{k-2}\|^2 - 
\]
\[\frac{\alpha}{2} \|\nabla f(x_{k-1})\|^2.
\]

Due to (\ref{PL}) we get
\[
f(x_k) - f^* + \frac{1-\alpha L}{2 \alpha} \|x_k - x_{k-1}\|^2 \; \leq  \; f(x_{k-1}) - f^* + \frac{\beta^2}{2 \alpha} \|x_{k-1} - x_{k-2}\|^2 -
\]
\[\alpha \mu (f(x_{k-1}) - f^*) =
\left(1 - \alpha \mu \right) \left((f(x_{k-1}) - f^*) + \frac{\beta^2}{2 \alpha (1 - \alpha \mu)} \|x_{k-1} - x_{k-2}\|^2 \right).
\]

Provided that
\[
\frac{\beta^2}{2 \alpha (1 - \alpha \mu)} \; \leq \; \frac{1-\alpha L}{2 \alpha};
\quad \quad
\beta \; \leq \; \sqrt{(1 - \alpha L)(1 - \alpha \mu)} ,
\]

we obtain
\[
V(x_{k+1}) \; \leq \; q V(x_{k}) \; \leq \; q^k V(x_1), \quad q = (1 - \alpha \mu) < 1.
\]

For convenience, we denote $\ \gamma = \frac{1-\alpha L}{2 \alpha}\ $ and thus 
\[
f(x_{k+1}) - f^* \; \leq \; f(x_{k+1}) - f^*  + \gamma \|x_{k+1} - x_{k}\|^2 = V_{k+1} \; \leq 
\]
\[q^k V_1 = q^k \left( f(x_1) - f^*  + \gamma \|x_1 - x_0\|^2 \right) = q^k \left( f(x_0) - f^* \right).
\]

\end{proof}

Global convergence of the Heavy Ball method is established here for function under condition (\ref{PL}) through Lyapunov function (\ref{Lyap_function}). The similar results on global convergence for more narrow class of strongly convex functions were obtained in paper \cite{Global}. 

\subsection{Adaptive algorithm}
In this section we will consider a general idea of choosing the optimal parameters for the accelerated gradient method. 

We consider the smooth and strongly convex problem (\ref{opt_pr}) focusing on the Heavy Ball method (\ref{mtsh}). In order to ensure the convergence of the method, it is necessary to select the parameters (\ref{al_mtsh}) correctly. Therefore, the values of strong convexity constant $\mu$ and the Lipschitz constant $L$ are required. Unfortunately, these constants are difficult and time-consuming to compute for a real problem. Moreover, as already discussed earlier, accelerated first-order methods are not guaranteed to be monotonic.

To sum up, the following situations should be distinguished:
\begin{itemize}
    \item non-monotone behavior (natural situation arising in first-order methods);
    \item mistake in parameter values.
\end{itemize}

Wide range of papers offer different options for adaptive restarting schemes dedicated to avoiding non-monotone behavior. By restart we denote starting the method from the very beginning with new parameters, where new initial condition is the current iteration. It is suggested to restart after a fixed number of iterations (“fixed restart”) or more efficiently after checking certain conditions of restart (“adaptive restart”). The papers \cite{boyd2014}, \cite{Restart_2}, \cite{Restart_1}  propose the different adaptive restart schemes and analysis of these techniques. For example, \cite{Restart_2} offers the following two adaptive restart techniques. They restart whenever:
\begin{enumerate}
    \item function scheme
    \[
    f(x_k) > f(x_{k-1});
    \]
    \item gradient scheme
    \[
    \nabla f(x_{k-1})^T(x_k - x_{k-1}) > 0.
    \]
\end{enumerate}

In this work, we offer an alternative to restart techniques in the form of the Lyapunov function $\ V(x) \ $. This function strictly decreases at each iteration. It assures that the method is stable. The algorithm converges to optimal value. We propose to monitor the values of the Lyapunov function $\ V(x) \ $  instead of the objective function $\ f(x) \ $ at each iteration.
We will restart the algorithm with new parameters, only if the value of the Lyapunov function starts to increase. 

Lyapunov scheme
\[
V(x_k) > V(x_{k-1}).
\]
We propose construction of the Lyapunov function (\ref{Lyap_function}) for discrete case:
\[
V(x_k) = f(x_k) + \frac{1-\alpha L}{2 \alpha} \|x_k - x_{k-1}\|^2.
\]

Note that the correct choice of the parameters implies only the knowledge of Lipschitz constant, which can be determined iteratively according to \cite{Nesterov_3_L}. It means, that we will check inequalities (\ref{13}) with additional term $\frac{\epsilon}{2} \ $, where $\epsilon > 0$ --- the required accuracy.
\[
f(x_{k}) \leq f(x_{k-1}) + \left\langle \nabla f(x_{k-1}), x_{k} - x_{k-1}\right\rangle + \frac{L}{2} \|x_{k} - x_{k-1}\|^2 + \frac{\epsilon}{2}
\]
As a result, we obtain the following adaptive algorithm:
\begin{algorithm}[H]
	\caption{Adaptive Heavy ball method with Lyapunov function}
	\begin{algorithmic}[1]
		\State \textbf{Input:} $f \in \mathscr F^{1,1}_{L} \, , \ x_0=y_0 \in R^n\, , \ L_0 > 0, \ 
        \alpha_0 \in \left( 0, \frac{1}{L_0} \right), \ 
        \beta_0 \in \left( 0, \sqrt{1 - \alpha_0 L_0}\right)$.
		\State \textbf{for} $k \geq 0$ \textbf{do}
		\State $ \qquad x_{k+1} = x_k - \alpha_k \nabla f(x_k) + \beta_k \left(x_{k+1} - x_k \right)$
		\State $ \qquad $ \textbf{if} $ \ V(x_{k+1}) \ > \ V(x_k) $
		\State $ \qquad \qquad $ \textbf{then} $ \ L_{k+1} = 2L_k, \ \alpha_{k+1} \in \left( 0, \frac{1}{L_{k+1} } \right), \ 
        \beta_{k+1}  \in \left( 0, \sqrt{1 - \alpha_{k+1}  L_{k+1} }\right)$ 
		\State $ \qquad \qquad $ \textbf{else} $ \ L_{k+1} = L_k, \ \alpha_{k+1} = \alpha_k, \ \beta_{k+1} = \beta_k$ 
	\end{algorithmic}
\end{algorithm}

\section{Conclusion}

In this work, attention has been paid to behavior of second-order difference equations' solutions and their connection with accelerated gradient methods' convergence. Firstly, considering linear difference equations we derived the initial conditions causing peaking effects. In the next section these results were applied to the analysis of the Heavy Ball method in case of a quadratic objective function and optimal parameters $\alpha$ and $\beta$. Then, we moved on to discussing the non-monotonic behaviour of the method for strongly convex and smooth functions $ f(x) \in \mathscr F^{1,1}_{L,\mu} \ $. Finally, the concept of the Lyapunov function for method's control was suggested.

The future work implies expanding the notion of Lyapunov function to other classes of objective functions and developing an adaptive algorithm with a better convergence rate. The obtained results are supposed to be applied to numerous unconstrained optimization problems, arising in power system engineering, deep-learning and other fields.

\section*{Funding}
Financial support for this work was provided by the Russian Science Foundation, project no. 16-11-10015.

%
%

\end{document}